\documentclass[a4paper,intlimits,12pt,reqno]{amsart}
\usepackage[T1]{fontenc}
\usepackage[utf8]{inputenc}

\usepackage{latexsym,amsthm,amsmath,amssymb,mathrsfs,mathtools}
\usepackage{tikz}
\usetikzlibrary{arrows,calc,patterns}
\usepackage{wrapfig}

\usepackage[mathscr]{eucal}

plus 6pt minus 12pt
\newcommand{\Sph}{\mathbb S}
\newcommand{\cH}{\mathcal H}

\newcommand{\bbbr}{\mathbb R}

\newcommand{\bbbq}{\mathbb Q}

\newcommand{\eps}{\varepsilon}

\def\unp{\operatorname{Unp}}

\def\dist{\operatorname{dist}}

\newtheorem{theorem}{Theorem}
\newtheorem*{theorem*}{Theorem}
\newtheorem{lemma}[theorem]{Lemma}

\newtheorem{proposition}[theorem]{Proposition}

\theoremstyle{definition}
\newtheorem{remark}[theorem]{Remark}

\def\mvint_#1{\mathchoice
          {\mathop{\vrule width 6pt height 3 pt depth -2.5pt
                  \kern -9pt \intop}\limits_{\kern -3pt #1}}%
          {\mathop{\vrule width 5pt height 3 pt depth -2.6pt
                  \kern -6pt \intop}\nolimits_{#1}}%
          {\mathop{\vrule width 5pt height 3 pt depth -2.6pt
                  \kern -6pt \intop}\nolimits_{#1}}%
          {\mathop{\vrule width 5pt height 3 pt depth -2.6pt
                  \kern -6pt \intop}\nolimits_{#1}}}

\usepackage{setspace}

\title[Ambiguous locus]{On an old theorem of Erd\"os about ambiguous locus}

\author{Piotr Haj\l{}asz}
\address{Piotr Haj\l{}asz,\newline \indent Department of Mathematics, University of Pittsburgh, \newline \indent 301 Thackeray Hall, Pittsburgh,
Pennsylvania 15260}
\email{hajlasz@pitt.edu}
\thanks{The author was supported by the NSF grant DMS-2055171.}


\begin{document}

\sloppy

\begin{abstract}
Erd\"os proved in 1946 that if a set $E\subset\bbbr^n$ is closed and non-empty, then the set, called ambiguous locus or medial axis, of points in $\bbbr^n$ with the property that the nearest point in $E$ is not unique, can be covered by countably many surfaces, each of finite  $(n-1)$-dimensional measure. We improve the result by obtaining a new regularity result for these surfaces in terms of convexity and $C^2$ regularity. 
\end{abstract}

\maketitle

Given a closed set $\varnothing\neq E\subset\bbbr^n$, let $\unp(E)$ be the set of all points $x\in\bbbr^n$ for which there is a unique point $y\in E$ nearest to $x$.
Clearly $E\subset\unp(E)$.
If we denote this nearest point by $\pi(x):=y$, the mapping $\pi:\unp(E)\to E$ is called the {\em metric projection}. In order to understand the properties of this mapping it is important to understand the structure of the set $\bbbr^n\setminus\unp(E)\subset\bbbr^n\setminus E$, where we lack uniqueness of the metric projection. This set is often called the {\em ambiguous locus} of the metric projection. It is also called the {\em medial axis} or the {\em skeleton} of $\bbbr^n\setminus E$.
Zamfirescu \cite{zam} (see also \cite[4A]{fremlin}, \cite{zhivkov}) proved that for most compact sets $\varnothing\neq E\subset\bbbr^n$ (in the Baire category sense with respect to the Hausdorff distance on the spaces on compact subsets of $\bbbr^n$), the set $\unp(E)$ has empty interior, meaning that the set of points in $\bbbr^n$ without a unique nearest point in $E$ is dense.  
On the other hand, it is known that the Lebesgue measure of the set $\mathbb{R}^n\setminus\unp(E)$ equals zero, $|\mathbb{R}^n\setminus\unp(E)|=0$.
This result is due to Erd\"os \cite{erdos2}. For a simple folklore proof (different from that in \cite{erdos2}), see Lemma~\ref{T2} and Remark~\ref{rem2} below.

Erd\"os \cite{erdos}, proved however, a much stronger result: {\em The set $\mathbb{R}^n\setminus\unp(E)$ is contained in the sum of countably many surfaces of finite $(n-1)$-dimensional measure.} His proof is based on Roger's \cite{roger} proof of the {\em contingent theorem} (see also \cite[pp. 264-266 and 304-307]{saks} and \cite[Section~2.1.8]{mat}). 
Fifty years later, Erd\"os' result was rediscovered by Fremlin \cite[Theorem~1G]{fremlin} with a different proof, but the author was not aware of the work of Erd\"os.
For more results about the structure of the set $\mathbb{R}^n\setminus\unp(E)$, see \cite{ac} and references therein. Another interesting and related reference is \cite{BH}.

In Theorem~\ref{T1} which is the main result of the paper, we substantially improve Erd\"os' result by showing convexity and $C^2$ regularity properties of the surfaces covering the set $\bbbr^n\setminus\unp(E)$. Our argument is different from that of Erd\"os. 
To the best of my knowledge, Theorem~\ref{T1} is new. While, it can be easily deduced from the results existing in the literature (as we do it here), I believe, the result is of a substantial interest and it deserves a clear, self-contained, and easy to read proof. In addition to Theorem~\ref{T1}, I believe, the proof of Theorem~\ref{T10} is of independent interest as I explained it in Remark~\ref{rem}.

We say that the set 
\begin{equation}
\label{eq12}
G=\{x\in\bbbr^n:\, x_i=f(x_1,\ldots,x_{i-1},x_{i+1},\ldots,x_n)\},
\end{equation}
where $1\leq i\leq n$ and $f:\bbbr^{n-1}\to\bbbr$ is continuous, is a  {\em $C^2$-graph} if $f\in C^2$ and it is a $(c-c)$-graph if $f=g-h$ is the difference of convex functions $g,h:\bbbr^{n-1}\to\bbbr$.
If the convex functions $g,h$ are of class $C^{2}$, we say that $G$ is a {\em $C^{2}-(c-c)$-graph}.

\begin{theorem}
\label{T1}
For any closed set $E\subset\bbbr^n$ we have
\begin{enumerate}
\item[(a)]
The set $\bbbr^n\setminus\unp(E)$ can be covered by countably many $(c-c)$-graphs.
\item[(b)]
There are countably many $C^{2}-(c-c)$-graphs $\{ G_j\}_{j=1}^\infty$ such that
$$
\cH^{n-1}\Big((\bbbr^n\setminus\unp(E))\setminus\bigcup_{j=1}^\infty G_j\Big)=0,
$$
where $\cH^{n-1}$ stands for the Hausdorff measure.
\end{enumerate}
\end{theorem}
\begin{remark}
Since convex functions are locally Lipschitz continuous \cite[Theorem~41D]{RV}, 
compact subsets of $(c-c)$-graphs have finite $(n-1)$-dimensional Hausdorff measure and hence the
$(n-1)$-dimensional Hausdorff measure of $\bbbr^n\setminus\unp(E)$ is $\sigma$-finite. 
Therefore, part (a) implies Erd\"os' result \cite{erdos} mentioned earlier.
\end{remark}
\begin{remark}
It follows from (b) that $\bbbr^n\setminus \unp(E)$ is $(\cH^{n-1},n)$-rectifiable of class $C^2$ in the sense of \cite{anzellottiS}.
\end{remark}

The paper is organized as follows. In Section~\ref{convex} we collect basic facts about convex functions. In Section~\ref{ludek} we prove a special case of a result of Zaj\'{\i}\v{c}ek \cite{zajicek}, and in Section~\ref{trzy} we prove Theorem~\ref{T1}.

\section{Convex functions}
\label{convex}
Let $f:\bbbr^n\to\bbbr$ be a continuous function. The {\em subdifferential of $f$ at $x$}, denoted by $\partial f(x)$,  is the set of all $v\in\bbbr^n$ such that
\begin{equation}
\label{eq4}
f(x+h)\geq f(x)+\langle v,h\rangle
\quad
\text{for all $h\in\bbbr^n$.}
\end{equation}
The geometric interpretation is that each $v\in\partial f(x)$ defines
a hyperplane passing through $(x,f(x))$ such that the graph of $f$ is above that hyperplane. Such hyperplanes are called {\em supporting hyperplanes} of the graph of $f$ at $(x,f(x))$.
If $\partial f(x)$ contains more than one vector, it means that there is more than one supporting hyperplane at $(x,f(x))$.

\begin{lemma}
\label{T6}
If $f:\bbbr^n\to\bbbr$ is convex, then for every $x\in\bbbr^n$, $\partial f(x)\neq\varnothing$. That is at every point, there is at least one supporting hyperplane. If in addition $f$ is differentiable at $x$, then $\partial f(x)=\{\nabla f(x)\}$ so in that case the tangent hyperplane  is a unique supporting hyperplane.
\end{lemma}

\begin{proof}
It easily follows from the definition of the derivative and convexity, that if $f$ is differentiable at $x$, then there is a unique supporting hyperplane defined by $\nabla f(x)$ i.e., $\partial f(x)=\{\nabla f(x)\}$. Thus it remains to prove existence in the general case.

For $k=1,2,\ldots$, let $p_k$ be the point on the graph of $f$ nearest to $q_k=(x,f(x)-k^{-1})$. Let $H_k$ be the hyperplane orthogonal to the segment $[p_k,q_k]$ and passing through its midpoint. It follows from the convexity of $f$ that the graph of $f$ lies above the hyperplane $H_k$. Since the unit normal vectors to $H_k$ belong to the compact unit sphere $\nu_k\in\Sph^{n-1}$, we can select a convergent sequence $\nu_{k_i}\to\nu\in\Sph^{n-1}$ and it is easy to see that the hyperplane normal to $\nu$ and passing through $(x,f(x))$ is a supporting one. 
\end{proof}

\begin{lemma}
\label{T4}
If  $f:\bbbr^n\to\bbbr$ is convex and partial derivatives $\partial f/\partial x_i$, $i=1,2\ldots,n$ exist at a point $x\in\bbbr^n$, then $f$ is (Fr\'echet) differentiable at $x$.
\end{lemma}
For a proof see \cite[Theorem~42D]{RV}.
This lemma follows from Jensen's inequality: any vector $h\in\bbbr^n$ can be expressed as a convex combination of vectors parallel to coordinate axes and this along with the Jensen inequality applied to the convex function $\varphi(h):=f(x+h)-f(x)-\langle \nabla f(x),h\rangle$ allows us to show that $\varphi(h)$ converges to $0$ as $o(|h|)$. It might be more rewarding to fill missing details as an exercise rather than to read the proof from~\cite{RV}.

One sided partial derivatives will be denoted by 
\begin{equation}
\label{eq2}
\frac{\partial^\pm f}{\partial x_i}(x)=\lim_{t\to 0^\pm}\frac{f(x+te_i)-f(x)}{t}\, .
\end{equation}
The next lemma easily follows from the monotonicity of secants of a convex function in one variable.
\begin{lemma}
\label{T8}
If $f:\bbbr^n\to\bbbr$ is convex, then one-sided partial derivatives \eqref{eq2} exist at every point $x\in\bbbr^n$ and $\partial^-f(x)/\partial x_i\leq \partial^+f(x)/\partial x_i$.
Moreover, for any $x\in\bbbr^n$, and $1\leq i\leq n$ 
\begin{equation}
\label{eq13}
f(x+te_i)\geq f(x)+s t\ 
\text{for all $t\in\bbbr$ and all $s$ satisfying}
\
\frac{\partial^-f}{\partial x_i}(x)\leq s\leq 
\frac{\partial^+f}{\partial x_i}(x).
\end{equation}
\end{lemma}

Thus a convex function $f$ is {\em not} differentiable at $x$ if and only if there is $i\in\{1,2,\ldots,n\}$ such that
\begin{equation}
\label{eq6}
\frac{\partial^-f}{\partial x_i}(x)
<\frac{\partial^+f}{\partial x_i}(x).
\end{equation}
Indeed, according to Lemma~\ref{T4}, $f$ is differentiable at $x$ if and only if partial derivatives exist and that is equivalent to equality of all one-sided partial derivatives.

Therefore, if $A\subset\bbbr^n$ is the set of points where a convex function $f$ is not differentiable, then
\begin{equation}
\label{eq7}
A=\bigcup_{i=1}^n\bigcup_{\alpha<\beta\atop \alpha,\beta\in\bbbq}
\underbrace{\Big\{x:\, \frac{\partial^-f}{\partial x_i}(x)\leq \alpha<\beta\leq\frac{\partial^+f}{\partial x_i}(x)\Big\}}_{A_{\alpha,\beta}^i}.
\end{equation}

We say that a convex function $f:\bbbr^n\to\bbbr$ {\em is coercive} if $f(x)\to\infty$ as $|x|\to\infty$, and $f$ is {\em strongly convex} if $f(x)-\mu|x|^2$ is convex for some $\mu>0$. 
Clearly, if $f$ is convex, then $f(x)+|x|^2$ is strongly convex.

If $f$ is strongly convex, then it is coercive in the following stronger sense:
\begin{equation}
\label{eq3}
\lim_{|x|\to\infty} (f(x)-\ell(x))=\infty
\quad
\text{for any linear function $\ell:\bbbr^n\to\bbbr$.}
\end{equation}
Indeed, $f(x)-\mu|x|^2$ is convex and hence an affine function (supporting hyperplane) bounds it from below, $f(x)-\mu|x|^2\geq A(x)$  so $f(x)-\ell(x)\geq A(x)+\mu|x|^2-\ell(x)\to\infty$ as $|x|\to\infty$.

\begin{lemma}
\label{T5}
If $f:\bbbr^k\times\bbbr^\ell\to\bbbr$ is convex and coercive, then
$F(x):=\inf_{y\in\bbbr^\ell} f(x,y)$ defines a convex function $F:\bbbr^k\to\bbbr$.
\end{lemma}
\begin{proof}
Let $x_1,x_2\in\bbbr^k$ and $\lambda\in [0,1]$. Then for any $y_1,y_2\in\bbbr^\ell$ we have
\begin{equation*}
\begin{split}
&F(\lambda x_1+(1-\lambda)x_2)\leq f(\lambda x_1+(1-\lambda)x_2,\lambda y_1+(1-\lambda)y_2)\\
&=
f(\lambda(x_1,y_1)+(1-\lambda)(x_2,y_2))
\leq
\lambda f(x_1,y_1)+(1-\lambda)f(x_2,y_2)
\end{split}
\end{equation*}
and the result follows upon taking the infima over $y_1\in\bbbr^\ell$ and $y_2\in\bbbr^\ell$.
\end{proof}

\section{A theorem of Lud\v{e}k Zaj\'{\i}\v{c}ek} 
\label{ludek}
The next result is a special case of a theorem of Zaj\'{\i}\v{c}ek \cite[Theorem~1]{zajicek}.
\begin{theorem}
\label{T10}
If $f:\bbbr^n\to\bbbr$
is convex, then the set of points where $f$ is not differentiable
is contained in a countable union of $(c-c)$-graphs.
\end{theorem}
\begin{remark}
\label{rem}
While our proof is almost the same as the original one, the result is a slight improvement of that of Zaj\'{\i}\v{c}ek, as Proposition~\ref{T11} and its proof provide a more detailed description of surfaces covering the non-differentiability sets $A_{\alpha,\beta}^i$ in the case of strongly convex functions.
Other reasons to include a proof are: (1) to make the paper self-contained; (2) the result \cite[Theorem~1]{zajicek} whose special case we prove here is more difficult to read due to its generality; (3) while the theorem of Zaj\'{\i}\v{c}ek has been cited many times, it is a good idea to provide an easy to read proof as it might contribute to popularization of this beautiful result.
\end{remark}

Since $f$ is non-differentiable at $x$ if and only if the strongly convex function $f(x)+|x|^2$ is non-differentiable at $x$, Theorem~\ref{T10} follows from the next result whose proof provides a more detailed description of the structure of the discontinuity sets $A_{\alpha,\beta}^i$.

We say that graphs of the form \eqref{eq12} are  graphs {\em in the direction of $x_i$.}

\begin{proposition}
\label{T11}
If $f:\bbbr^n\to\bbbr$ 
is strongly convex, then each of the sets $A_{\alpha,\beta}^i$ is contained in a $(c-c)$-graph in the direction of $x_i$. Therefore, the set $A$ is contained in a countable union of $(c-c)$-graphs.
\end{proposition}
\begin{proof}
Without loss of generality we may assume that $i=1$. 
For any $s\in\bbbr$, the function $f_s(x)=f(x)-sx_1$ is convex and coercive by \eqref{eq3}. Therefore, Lemma~\ref{T5} implies that the function
$$
g_s:\bbbr^{n-1}\to\bbbr,
\quad
g_s(x_2,\ldots,x_n):=\inf_{x_1\in\bbbr} f_s(x_1,x_2,\ldots,x_n)
$$
is convex. If $a\in A^1_{\alpha,\beta}$, then \eqref{eq13} yields
$$
f(a+te_1)\geq f(a)+\alpha t
\quad
\text{for all $t\in\bbbr$}
$$
or equivalently,
$$
f_\alpha(a+te_1)\geq f_\alpha(a)
\quad
\text{for all $t\in\bbbr$.}
$$
The last condition however, means that the function $x_1\mapsto f_\alpha(x_1,a_2,\ldots,a_n)$ attains minimum at $x_1=a_1$ so 
\begin{equation}
\label{eq9}
g_\alpha(a_2,\ldots,a_n)=f_\alpha(a_1,a_2,\ldots, a_n)=f(a_1,a_2,\ldots,a_n)-\alpha a_1.
\end{equation}
Similarly, inequality 
$$
f(a+te_1)\geq f(a)+\beta t
\quad
\text{for all $t\in\bbbr$}
$$
(also guaranteed by \eqref{eq13}) implies that
\begin{equation}
\label{eq10}
g_\beta(a_2,\ldots,a_n)=f(a_1,a_2,\ldots,a_n)-\beta a_1.
\end{equation}
Now \eqref{eq9} and \eqref{eq10} yield
$$
a_1=\frac{1}{\beta-\alpha}(g_\alpha(a_2,\ldots,a_n)-g_\beta(a_2,\ldots,a_n))
$$
and $A^1_{\alpha,\beta}$ is contained in the graph of the $(c-c)$-function
\begin{equation}
\label{eq11}
g(x_2,\ldots,x_n)=\frac{1}{\beta-\alpha}(g_\alpha(x_2,\ldots,x_n)-g_\beta(x_2,\ldots,x_n)).
\end{equation}
The proof is complete.
\end{proof}
\begin{remark}
Note that in general not all points of the graph of \eqref{eq11} belong to the set $A^1_{\alpha,\beta}$
as otherwise we would obtain a whole surface of points of discontinuity of the derivative and $f(x)=|x|$ has only one such point.
\end{remark}

\section{Proof of Theorem~\ref{T1}}
\label{trzy}

\begin{proof}[Proof of (a)]
In the proof we will need the following two lemmata. 
Erd\"os~\cite{erdos2}, proved that $|\bbbr^n\setminus\unp(E)|=0$, but Lemma~\ref{T2} provides a different proof of this fact, see
\cite[Proposition~2.4]{clarke}, \cite[Lemma~2.1]{CF}, \cite[Theorem~3.3]{EH}, \cite[Lemma~2.21]{BD}.
\begin{lemma}
\label{T2}
If a set $\varnothing\neq E\subset\bbbr^n$ is closed, and
the distance function $d(x)=\dist(x, E)$ is differentiable at $x\in\bbbr^n\setminus E$, then $x\in \unp(E)$.
\end{lemma}
\begin{remark}
\label{rem2}
Since $d$ is Lipschitz continuous, it is differentiable a.e. by the Rademacher theorem, and hence 
Lemma~\ref{T2} yields $|\bbbr^n\setminus\unp(E)|=0$. 
\end{remark}
\begin{proof}
Assume that $d$ is differentiable at $x\in\bbbr^n\setminus E$ and $p\in E$ is such that $|x-p|=d(x)$. We will prove that $p=x-d(x)\nabla d(x)$ and this will imply uniqueness of $p$.

It follows from the triangle inequality that if $y$ belongs to the interval with endpoints $x$ and $p$, $y\in [x,p]$, then $d(y)=|y-p|$. Therefore, the function $d$ decreases linearly with the slope $1$ along the segment $[x,p]$. Since $d$ is differentiable at $x$, the directional derivative at $x$ in the direction of the unit vector $v=(p-x)/|p-x|$ satisfies
\begin{equation}
\label{eq1}
\langle\nabla d(x), v\rangle= D_vd(x)=-1.
\end{equation}
On the other hand, $d$ is $1$-Lipschitz so $|\nabla d(x)|\leq 1$ and hence equality \eqref{eq1} implies that
$$
\nabla d(x)=-v=\frac{x-p}{d(x)}
\quad
\text{so}
\quad
p=x-d(x)\nabla d(x).
$$
\end{proof}
The next lemma was observed by Asplund \cite[p.\ 235]{asplund}.
\begin{lemma}
\label{T3}
If $\varnothing\neq E\subset\bbbr^n$ is closed, and $d(x)=\dist(x,E)$, then
the function $f:\bbbr^n\to\bbbr$ defined by $f(x)=|x|^2-d(x)^2$ is convex.
\end{lemma}
\begin{proof}
We have
$$
f(x)=|x|^2-\inf_{y\in E}|x-y|^2=|x|^2+\sup_{y\in E}(-|x-y|^2)=\sup_{y\in E} \big(2\langle x,y\rangle -|y|^2\big).
$$
Therefore,
$f$ is a supremum of a family of affine functions,
and hence it is convex.
\end{proof}

By Lemma~\ref{T2} the set $\bbbr^n\setminus\unp(E)$ is contained in the set 
$A_+$ of points where $d$ is strictly positive and non-differentiable. Since $d>0$ in $A_+$, the set $A_+$ is contained in the set where the convex function $f(x)=|x|^2-d(x)^2$ is non-differentiable and the result follows from Theorem~\ref{T10}.
\end{proof}

\begin{proof}[Proof of (b)]
It is well known that if $f:\bbbr^{n-1}\to\bbbr$ is convex, then for any $\eps>0$ there is a function $f_\eps\in C^2(\bbbr^{n-1})$ such that
the Lebesgue measure of the set where the two functions differ satisfies 
$|\{x:\, f(x)\neq f_\eps(x)\}|<\eps$. (In general,
the function $f_\eps\in C^2$ cannot be convex, see Example~1.9 and Proposition~1.10 in \cite{azagrah}.)
This is a consequence of the Aleksandrov theorem \cite{Alexandroff} about second order differentiability of convex functions which implies that a convex function satisfies assumptions of the $C^2$-Whitney extension theorem outside a set of measure less than $\eps$. 
While the idea is simple, the details are rather difficult, see \cite{Alberti}, \cite[Corollary~1.5]{azagrah}, \cite[Proposition~A1]{EvansGangbo}, \cite{Imomkulov}. 
This and part (a) easily imply that the set $\bbbr^n\setminus\unp(E)$ can be covered up to a set of $\cH^{n-1}$-measure zero by a sequence of graphs of $C^2$-functions.
One only needs to note that any $(c-c)$-function, $g:\bbbr^{n-1}\to\bbbr$ is locally Lipschitz and hence for a set $A\subset\bbbr^{n-1}$ of measure zero, the corresponding set on the graph of $g$ has vanishing Hausdorff measure $\cH^{n-1}$.

Therefore, it remains to show that if $f\in C^2$, then on every bounded set, $f$ coincides with the difference of two convex functions of class $C^{2}$. 

Given $R>0$, let $\varphi$ be a compactly supported smooth function that equals $1$ for $|x|\leq R$. Then $\varphi f\in C^2$ has bounded second order derivatives so there is $C_R>0$ such that the matrix $D^2(\varphi(x) f(x)+ C_R|x|^2)=D^2(\varphi f)+2C_R I$ is positive definite and hence 
the function $\varphi(x) f(x)+ C_R|x|^2$ is convex and of class $C^{2}$.  Now,
$f(x)=(\varphi(x)f(x)+C_R |x|^2)-C_R|x|^2$, $|x|\leq R$, represents $f$ for $|x|\leq R$ as a difference of convex functions of class $C^{2}(\bbbr^n)$.
This completes the proof.
\end{proof}

\noindent
{\bf Acknowledgement.} The author would like to thank Daniel Azagra and anonymous referee for helpful comments.

\end{document}